\documentclass{amsproc}

\usepackage[all]{xy}
\usepackage{amssymb}
\usepackage{amscd}
\usepackage{amsfonts}

\theoremstyle{plain}
\newtheorem{introThm}{Theorem}

\newtheorem{Thm}{Theorem}[section]
\newtheorem{Lem}[Thm]{Lemma}

\newtheorem{Conj}[Thm]{Conjecture}

\theoremstyle{definition}

\newtheorem{Rem}[Thm]{Remark}

\newcommand{\ad}{{\mathrm{ad}}}

\newcommand{\map}{{\mathrm{map}}}
\newcommand{\Der}{{\mathrm{Der}}}
\newcommand{\Ad}{{\mathrm{Ad}}}
\newcommand{\Pad}{{\mathrm{Pad}}}

\newcommand{\Dirlim}{\varinjlim}
\newcommand{\Invlim}{\varprojlim}

\newcommand{\Q}{{\mathbb{Q}}}
\newcommand{\PPP}{{\mathbb{P}}}

\newcommand{\G}{{\mathcal{G}}}

\newcommand{\aut}{\text{\rm aut}(G) }

\def\:{\colon\!}

\def\aut{\text{\rm aut}}

\begin{document}

\title[Group-like function spaces]
{Localization of  grouplike function and section spaces with compact domain} 

\author{Claude L.~Schochet}
\address{Department of Mathematics,
     Wayne State University,
     Detroit MI 48202}
\email{claude@math.wayne.edu}

\author{Samuel B. Smith}
\address{Department of Mathematics,
     Saint Joseph's University,
     Philadelphia PA 19131}
\email{smith@sju.edu}
\keywords{gauge group,   fibrewise self-equivalences,  fibrewise groups,   fibrewise localization, rational homotopy theory }
\subjclass[2000]{55P60,55P62, 55R70,55R10, 55Q52}
\begin{abstract}
By recent results in  \cite{KSS},   the standard localization theory for function and section spaces due to Hilton-Mislin-Roitberg  and M\"{o}ller  extend outside the CW category to   the case of compact metric domain in the presence of a grouplike structure.    We study  applications in two cases  directly generalizing the gauge group of a principal bundle.      We prove an identity for the monoid  $\aut(\xi)$ of  fibre-homotopy self-equivalences of a Hurewicz fibration $\xi$  --- due to Gottlieb and Booth-Heath-Morgan-Piccinini in the CW category ---  in the  compact case. This leads to  an extended  localization result for  $\aut(\xi)$.   We also obtain an extended localization theory for groups of sections $\Gamma(\zeta)$ of a fibrewise group $\zeta$. 
We give two  applications in rational homotopy theory.  
%
  
  \end{abstract}

\maketitle

 \section{Introduction}
We study two generalizations of the   homotopy classification problem for gauge groups. 
    Let $X$ be a  space, $G$ a connected CW topological group and  $h \colon X \to BG$ a map.  
 The  {\em gauge group}  $\G(P)$   corresponding to this data  is   the topological group of $G$-equivariant self-maps  of $P \colon E \to X$  where $P$ is the principal $G$-bundle induced by $h$.   Fixing $G$ and $X$, the      problem  is that of   determining the  number of distinct homotopy types or, alternately,  the number of  H-homotopy types  corresponding to maps in $[X, BG].$   
 Complete results in special cases are given by  Kono \cite{K}, Crabb-Sutherland \cite{CS} and Kono-Tsukuda \cite{KT} among many others.
 
 After rationalization,  the gauge group classification problem admits a complete solution with considerable generality.
  By \cite[Theorem D]{KSS}, when  $X$ is  a compact metric space and $G$ is a homotopy finite, connected, CW group,    the rationalization of  the connected component     of the identity of $\G(P)$ is H-commutative and independent of the classifying map $h 
 \colon X \to BG$.
 When $X$ is actually a finite CW complex, this result    may be deduced  from   an identity for the gauge group due to Gottlieb, (\ref{eq:Gottlieb}), below, combined with the standard localization theory for function   spaces due to    Hilton-Mislin-Roitberg \cite{HMR} (see \cite[Theorem 5.7]{KSS}).   Alternately,  the result for $X$ finite CW follows from    corresponding localization results for section spaces due to M{\" o}ller  \cite{Mol} and a result of Crabb-Sutherland \cite[Proposition 2.2]{CS}.     
The proof   for $X$ compact metric in \cite{KSS} entails an  extension of    the  localization results of Hilton-Mislin-Roitberg  and M{\" o}ller to the case of non-CW domain.  
  In this paper, we   expand on this analysis to develop  corresponding extended    localization theories for  
two   natural generalizations of the gauge group.

 The two generalizations  we consider are related  to   two  basic identities for the gauge group.   First, let $\map(X, Y)$ denote the space of all continuous maps with the compact-open topology. Let  $\map(X, Y; f)$
denote  the path-component of a given map $f \colon X \to Y.$  Then,  for $X$ a finite CW complex,  there is an $H$-equivalence
 \begin{equation} \label{eq:Gottlieb}  \G(P) \simeq \Omega \map(X, BG; h) \end{equation}
 (see \cite{Got} and  \cite[Proposition 2.4]{AB}).  
 Second,  let $\Ad(P) \colon E \times_G G^{\ad} \to X$ denote  the adjoint bundle.  Here the total space  is the quotient of the product
by the diagonal action where $G^{\ad} = G$ is a left-$G$ space via the adjoint action 
and the projection is  induced by the projection $E \to X.$ This is an  non-principal $G$-bundle. We have an isomorphism \begin{equation} \label{eq:Ad} \G(P) \cong \Gamma(\Ad(P)),
\end{equation}  where the latter space is the group of
 sections with fibrewise multiplication  (see \cite[p.539]{AB}).

 Now suppose given  a Hurewicz fibration $\xi \colon E \to X$. We assume here and throughout that $X$ has a distinguished, nondegenerate basepoint. By the fibre of $\xi$, we will mean the fibre over this basepoint.    We consider the monoid $\aut(\xi)$  consisting  of all fibre-homotopy self-equivalences of $\xi$  covering the identity of $X$ topologized as a subspace of  $\map(E, E)$.  The monoid $\aut(\xi)$ is a natural generalization of the gauge group to the fibre-homotopy setting. Unlike the gauge group,  the rational H-homotopy type of this monoid is generally a nontrivial invariant of the fibre-homotopy  theory of $\xi$  (see \cite{FLS}).

By  \cite[Theorem 3.3]{BHMP}, the identity 
 (\ref{eq:Gottlieb}) extends to 
 a corresponding identity for $\aut(\xi)$ when both  $X$ and the fibre $F$ are   finite CW complexes.  Our first main result extends this identity, in turn, to the 
 case $X$ is compact metric.  
   Recall the universal $F$-fibration, for $F$ finite CW,  may be identified, up to homotopy type, as a  sequence $F \to B\aut_*(F) \to B\aut(F)$  where $\aut(F)$ and $\aut_*(F)$ are the monoids of free and based homotopy self-equivalences of $F$.  The spaces    $B\aut(F)$   and $B\aut_*(F)$   are the Dold-Lashof classifying spaces for these  monoids.   (See    \cite{St,   Maymemoir} for the classification theory  and  \cite{Gottlieb} for the identification with Dold-Lashof \cite{DL}.)    
 
\begin{introThm}  \label{introthm:aut} Let $X$ be a compact metric space, $F$ a finite CW complex and
 $h \colon X \to B\aut(F)$ a map.  Let $\xi \colon E \to X$ be the corresponding
 $F$-fibration.  Then there is an H-equivalence
 $$\aut(\xi) \simeq \Omega \map(X, B\aut(F); h).$$
 \end{introThm}

As a consequence, we obtain an extended localization theory for $\aut(\xi).$  
 Let $\PPP$ be a collection of primes.    We say a space $Y$ is {\em nilpotent} if $Y$ is connected, has the homotopy type of a  CW complex  
 and has  a nilpotent homotopy system (see \cite[Definition II.2.1]{HMR}).   In this case, $Y$ admits
 a $\PPP$-localization $\ell_Y \colon Y \to Y_\PPP$ \cite[Theorem II.3A]{HMR}. 
 Given a map $f \colon X \to Y$ we write $f_\PPP = \ell
 _Y \circ f \colon X \to Y_\PPP.$
  Given a    monoid $G$ we write $G_\circ$ for the path component of the identity.  Given a space $Z$ with distinguished basepoint  we write
$\Omega_0 Z$ for the space of loops based at the basepoint.  For the function space $\map(X, Y;f)$ we assume $f$ is the basepoint.


\begin{introThm} \label{intro1} Let $X$ be a simply connected compact metric space, $F$ a finite CW complex and
 $h \colon X \to B\aut(F)$ a map.   Let $\xi \colon E \to X$ be the corresponding
 $F$-fibration. Then $\aut(\xi)_\circ$ is a nilpotent space  
and  we have an H-equivalence $$ (\aut(\xi)_\circ)_\PPP \simeq\Omega_\circ \map(X, (B\aut(F)_\circ)_\PPP; (\tilde{h})_\PPP ). $$
Here $B\aut(F)_\circ$  --- the Dold-Lashof classifying space of $\aut(F)_\circ$ --- is the universal cover of $B\aut(F)$ and $\tilde{h} \colon X \to B\aut(F)_\circ$ is the unique lifting of $h.$ 
\end{introThm}

The  second generalization we consider is  based on (\ref{eq:Ad}) in which  the gauge group corresponds to  a group of sections.   Recall $\zeta \colon E \to X$ is a {\em fibrewise group}  if there is a fibrewise map $ m \colon E \times_X   E \to E$ over
$X$, a section   $e \colon X \to E$ and a map  $i \colon E \to E$ over $X$ satisfying: (i) $m$ is associative, (ii) $e$ is a two-sided unit and (iii) $i$ is an inverse with respect to the maps $m$ and $e$.
The space of sections $\Gamma(\zeta)$ is then a topological group with the multiplication of sections induced by $m.$  More generally, relaxing   the group axioms to require identities 
only up to homotopy, $\zeta$ is a {\em fibrewise grouplike space} (\cite[p.62]{CJ}) and $\Gamma(\zeta)$ is a grouplike space. 
Our motivating example is  the adjoint bundle   $\Ad(P)$ of a principal $G$-bundle $P,$ as above.  Observe that, if $P \colon E \to X$ is classified by a map $h \colon X \to  BG,$
then $\Ad(P)$ is the pullback  by $h$ of the universal $G$-adjoint bundle $  EG \times_G G^{\ad} \to BG$ which is, in particular,  a CW fibration.  

Suppose generally that $\zeta \colon E \to X$ is a fibrewise grouplike space  with connected grouplike fibre $G$ and, further, that   $\zeta$  is the pullback of a CW fibration.  
In this case,
we may still identify a fibrewise $\PPP$-localization $\zeta \to \zeta_{(\PPP)}$  of $\zeta$ (see  the remarks preceding  Theorem \ref{thm:Gammalocalize}, below). 
Our  third main  result extends  \cite[Theorem 5.3]{Mol}   from the case the base space is   finite  CW   to the case of compact metric base in this context. 

\begin{introThm} \label{intro2} Let $\zeta \colon E \to X$ be a fibrewise grouplike space with connected, CW grouplike fibre $G$  and base $X$  a  compact metric space. Suppose $\zeta$    is the pullback of a CW fibration.  Then $\Gamma(\zeta)_\circ$ is a nilpotent space and the map $$\Gamma(\zeta)_\circ \to \Gamma(\zeta_{(\PPP)})_\circ$$ induced by a fibrewise $\PPP$-localization $\zeta \to \zeta_{(\PPP)}$
 is a $\PPP$-localization map. 
 \end{introThm}

\begin{Rem}  In many circumstances, the fibrewise $\PPP$-localization $\zeta_{(\PPP)}$ of a fibrewise grouplike space $\zeta \colon E \to X$ will itself be  a  fibrewise grouplike space over $X$ and the fibre map $\zeta \to \zeta_{(\PPP)}$ will be equivariant.  In this case, 
the equivalence of Theorem \ref{intro2} is actually an H-equivalence.   For example, this is the case  for the adjoint bundle   and, generally, when $\zeta$ is a CW fibration.  For general compact metric $X$, the fibrewise grouplike structure for  $\zeta_{(\PPP)}$ is not assured  due to the lack of uniqueness of fibrewise $\PPP$-localization
outside the CW category.
\end{Rem}

The paper is organized as follows. We prove  Theorems \ref{introthm:aut}-\ref{intro2} in Section \ref{sec:localize}. In Section  
\ref{sec:applications}, 
we apply our results to obtain two consequences.   
We prove that the rationalization of $\aut(\xi)_\circ$  is H-commutative and independent of the classifying map for   fibrations $\xi$
 with fibre satisfying a famous conjecture in rational homotopy theory.  (Theorem \ref{thm:F}).  This result  extends  \cite[Theorem 4]{FT}.  
 As an application of Theorem \ref{intro2}, we extend   \cite[Theorem F]{KSS} 
 on the classification of projective gauge groups (Theorem \ref{thm:P}). In Section \ref{sec:final}, we deduce based versions of  our main results.   
\newline \,
\newline
\noindent
{\bf Acknowledgements.}  We are grateful to our coauthors in  the papers 
\cite{FLS, KSS, LPSS} for their part in discovering these ideas.  We thank  Peter Booth for  very helpful discussions of his work.

\section{Localization of Function and Section Spaces} \label{sec:localize}
 The  standard results  on $\PPP$-localization of function  spaces are as follows.  
By Milnor \cite{Mil}, the path components of  $\map(X, Y)$ are of CW homotopy type when $X$ is a 
compact metric space and $Y$ is a CW complex.  
 By Hilton-Mislin-Roitberg,   when $X$ is finite CW and $Y$ is a nilpotent space then $\map(X, Y; f)$ is itself a nilpotent space and  $$\map(X, Y; f)_\PPP  \simeq \map(X, Y_\PPP; f_\PPP)$$ \cite[Theorem II.3.11]{HMR}.    Recall we write $f_\PPP = \ell_Y \circ f \colon X \to Y_\PPP$.  
 
 These    results all hold for spaces of basepoint-preserving functions,  as well.  In what immediately follows, we will focus    on the basepoint free case. We discuss the based case in Section \ref{sec:final}. The preceding results  
also hold   with alternate hypotheses: namely,   $X$  may be any CW complex when $Y$ is  a finite Postnikov piece (see \cite{HMRS}).  We will
not consider these alternate hypotheses here as they are not amenable to our  methods.


 
 The results of Hilton-Mislin-Roitberg were extended to  section spaces by  M\"{o}ller \cite{Mol} (see also, Scheerer \cite{Scheerer2}). Let $\xi \colon E \to X$
 be a   fibration of connected CW complexes with $X$   finite.    Let $F$ be the fibre over a basepoint of $X$. If $F$ is a nilpotent space 
then   $\xi$  then admits  a fibrewise $\PPP$-localization which is a fibre map $\xi \to \xi_{(\PPP)}$ 
over $X$ inducing $\PPP$-localization on the fibres \cite{May}.
By \cite[Theorem 5.3]{Mol}, under these hypotheses, the component of  $\Gamma(\xi; s)$ corresponding to a given section $s \colon X \to E$ 
is  a nilpotent space and    
$$\Gamma(\xi; s)_\PPP  \simeq \Gamma(\xi_{(\PPP)}; s_0)$$
where $s_0$  is the section of $\xi_{(\PPP)}$ induced by $s.$

  Our extensions of these localization results  depend on    a classical result of Eilenberg-Steenrod in \cite{ES}. Let  $X$ be a compact metric space. By \cite[Theorem X.10.1]{ES}, there is an inverse system of finite complexes $X_j$ with structure maps $g_{ij} \colon X_j \to X_i$ for $ i \leq j$ and compatible maps $g_j \colon X\to X_j$ such that the induced map  $$g\colon X \to  \Invlim_j X_j$$ is a homeomorphism.   Further, given any map $f \colon X \to Y$ for $Y$ a CW complex, there is an index $m$ and a map $f_m \colon X_m \to Y$ such that $ f$ is homotopic to   $f_m \circ g_m.$    \cite[Theorem X.11.9]{ES}.  
  
We apply this result  to study the function space $\map(X, Y;f)$, as follows. Choose $m$ for $f \colon X \to Y$ as above. Given  $j \geq m$ define $f_j \colon X_j \to Y$ by setting
$f_j = f_m \circ g_{mj}.$  
Restricting to indices $ j \geq m$,  we obtain  a direct
system $\map(X_j ; Y; f_j)$ with structure maps $(g_{ij})^* \colon \map(X_i, Y; f_i) \to \map(X_j, Y; f_j)$    and compatible maps $$(g_j)^* \colon   \map(X_j, Y; f) \to \map(X, Y; f_j)$$ both  induced by precomposition. 
We have:

\begin{Thm} \label{thm:mapdirect} {\em (\cite[Theorem 6.4]{KSS})} Let $X$ be a compact metric space and $Y$ a CW complex.  Then, for all $n \geq 1$,  the maps $(g_j)^*$ above induce an isomorphism 
$$ \Dirlim_j \pi_n(\map(X_j, Y; f_j))  \cong  \pi_n(\map(X, Y;f)).
$$\qed
\end{Thm}

This result leads to an extension of the Hilton-Mislin-Roitberg
result, mentioned above, provided the space $\map(X, Y; f)$ is known {\em a priori} to be a nilpotent space \cite[Theorem 7.1]{KSS}. Essentially the same proof yields the following:

\begin{Thm} \label{thm:maplocalize} Let $X$ be a compact metric space. Let $Y$ be a nilpotent space with $\PPP$-localization $\ell_Y \colon Y \to Y_\PPP$. Let $f \colon X \to Y$ be a given map.
 Then $\ell_Y$ induces an $H$-equivalence
$$ \left(\Omega_\circ \map(X, Y; f) \right)_\PPP \simeq\Omega_\circ\map(X, Y_\PPP; f_\PPP).$$ 
\end{Thm}
\begin{proof}
Consider the commutative square
$$
\xymatrix{
\Dirlim_j \pi_n(\map(X_j,Y;f_j))  \ar[r]^>>>>>>>>{\cong} \ar[d]  & \pi_n(\map(X,Y;f))  \ar[d] \\
\Dirlim_j \pi_n(\map(X_j,Y_{\PPP}; (f_j)_\PPP )) \ar[r]^>>>>>>{\cong}  & \pi_n(\map(X,Y_{\PPP};f_\PPP)) \\
}
$$
with vertical maps   induced by $\ell_Y$.  That the  left vertical map is a $\PPP$-localization map follows from  \cite[Theorem II.3.11]{HMR}. Thus the right vertical map is a $\PPP$-localization map,  as well.  Looping, we see that $\ell_Y$ induces a weak H-equivalence $\left(\Omega_\circ \map(X, Y; f) \right)_\PPP \simeq\Omega_\circ\map(X, Y_\PPP; f_\PPP)$
and so an honest H-equivalence since the spaces are CW, again by   \cite{Mil} . 
\end{proof}

Next let $\xi \colon E \to X$ be a Hurewicz fibration over   a compact metric space with a fixed section $s \colon X \to E.$  
Suppose $\xi$ is the pullback of a CW fibration $\xi_0 \colon E_0 \to B$ for some 
map $ h \colon X \to B.$   Then $\Gamma(\xi; s)$ is of CW homotopy type \cite[Proposition 3.2]{KSS} Given an inverse system of finite complexes $X_j$ for $X$ with
compatible maps $g_j \colon X \to X_j$ as above, choose an index $m$ so that $h \colon X \to B$ factors as $h = h_m \circ g_m$ for some map $h_m \colon X_m \to B.$  Given $j \geq m$, write $\xi_j \colon E_j \to X_j$ for the pullback of the fibration $\xi_0$ via the
map $h_j = h_m \circ g_{mj} \colon X_j \to B$.   Restricting again to indices $j \geq m$ gives a direct system
of spaces of sections $\Gamma(\xi_j)$  with structure maps $\gamma_{ij} \colon \Gamma(\xi_i) \to \Gamma(\xi_j)$  and compatible maps $\gamma_j \colon \Gamma(\xi_j) \to \Gamma(\xi)$ induced by   $g_{ij}$ and $g_j$  for $j \geq i \geq m.$ 

  By  \cite[Theorem 6.5]{KSS}, with this set-up the induced map $\gamma \colon \Dirlim_{j} \Gamma(\xi_j) \to \Gamma(\xi)$ gives  a bijection
$$\Dirlim\pi_0(\Gamma(\xi_j)) \cong \pi_0(\Gamma(\xi))$$
of pointed sets. 
Thus we may choose an index $m' \geq m$ and a section $s_{m'}$ of $\xi_{m'}$
such that $s_{m'}$ induces a section homotopic to $s$ via $\gamma_{m'}$. Restricting now to indices $j \geq m'$, let $s_j$ be the section induced on $\xi_j$ by $s_{m'}.$  We obtain a direct system of connected spaces 
$\Gamma(\xi_j; s_j)$ with compatible maps $\gamma_j \colon \Gamma(\xi_j; s_j) \to \Gamma(\xi; s)$
   for $j   \geq m'.$  Quoting the theorem for $n \geq 1$, we have:
\begin{Thm} \label{thm:Gammalimit} {\em (\cite[Theorem 6.5]{KSS})} Let $\xi \colon E \to X$ be a Hurewicz fibration over $X$ a compact metric space with a fixed section $s \colon X \to E.$  
Suppose $\xi$ is the pullback of a CW fibration.  Then, for all $n \geq 1,$
the maps $\gamma_j$ above  induce an isomorphism
$$ \Dirlim_j \pi_n(\Gamma(\xi_j; s_j)) \cong \pi_n(\Gamma(\xi; s)).$$
 \qed
\end{Thm}

By the work of May \cite{May}, a   fibration $\xi \colon E \to X$ of   CW complexes
with nilpotent fibre $F$ admits a {\em fibrewise $\PPP$-localization} which is a fibration
$\xi_{(\PPP)} \colon E_0 \to X$ and a map $g \colon E \to E_0$ over $X$ such that 
$g$ induces  $\PPP$-localization  $F \to F_\PPP$ on fibres. This fibrewise $\PPP$-localization of $\xi$ is unique up to fibre-homotopy equivalence by Llerena \cite[Theorem 6.1]{Ll}.

We may directly extend 
this  construction    to  non-CW fibrations  which are nevertheless the pullback of   an appropriate CW fibration. Specifically,  let $\xi \colon E \to X$ be the pullback of a
fibration $\xi_0 \colon E_0 \to B$ of  CW complexes with nilpotent fibre
via a map $h \colon X \to B.$ We take $\xi_{(\PPP)} = h^{-1}((\xi_0)_{(\PPP)}),$ the pullback of the fibrewise $\PPP$-localization of $\xi_0$.  We note that uniqueness is no longer assured.

\begin{Thm} \label{thm:Gammalocalize} Let $\xi \colon E \to X$ be a fibration of  spaces with nilpotent
fibre and compact metric base. Suppose  $\xi$ is a pullback of a CW fibration
and that, for some section $s$ of $\xi,$ the component $\Gamma(\xi; s)$ is a nilpotent space.  Let $\xi \to \xi_{(\PPP)}$ be a fibrewise $\PPP$-localization and $s_0$ the induced section.  Then
$$\Gamma(\xi; s)_\PPP \simeq \Gamma(\xi_{(\PPP)}; s_0).$$
\end{Thm}
\begin{proof} Choose an inverse system $X_j$ of finite complexes for $X$ and let $\xi_j$ be the corresponding   fibrations over $X_j$ with fibre $F$ and compatible sections $s_j$.  Consider the commutative diagram, as in the proof of Theorem \ref{thm:maplocalize}:
$$\xymatrix{
\Dirlim_j \pi_n(\Gamma(\xi_j;s_j))  \ar[r]^>>>>>>>>{\cong} \ar[d]  & \pi_n(\Gamma(\xi; s))  \ar[d] \\
\Dirlim_j  \pi_n(\Gamma((\xi_j)_{(\PPP)};(s_j)_0)) \ar[r]^>>>>>>{\cong}  &  \pi_n(\Gamma(\xi_{(\PPP)}; s_0))\\ 
}$$
The horizontal maps are isomorphisms by Theorem \ref{thm:Gammalimit}. The left vertical map is  $\PPP$-localization by     \cite[Theorem 5.3]{Mol}. Thus the right 
vertical map is a $\PPP$-localization, as well.
\end{proof}

 Theorem \ref{thm:Gammalocalize} implies:  

\begin{proof}[Proof of Theorem \ref{intro2}]
Recall we are assuming $\zeta \colon E \to X$ is  fibrewise grouplike  over a compact metric space $X$ with fibre $G$ a connected CW grouplike space. Further, we assume $\zeta$ is the pullback of 
CW fibration.   By \cite[Proposition 3.2]{KSS}, $\Gamma(\zeta)_\circ$ is of CW type and so a
nilpotent space since grouplike.  The result is thus a direct consequence of Theorem \ref{thm:Gammalocalize}. 
\end{proof}

We now turn to the proof of Theorem \ref{intro1}. Fix a compact metric space $X$ and a finite  CW complex $F.$  We write $\xi_\infty \colon  E_\infty \to B\aut(F)$ for the universal $F$-fibration. (Here $E_\infty \simeq B\aut_*(F).$)    Given a map $h \colon X \to B\aut(F),$ let $\xi \colon E \to X$ be the pullback. In  \cite{BHMP}, the authors  define a fibration $$ \xi \, \Box_1 \xi_\infty \colon E \, \Box E_\infty \to X$$ in which the  total space is the set of all homotopy equivalences $f_{a,b} \colon E_a \to (E_\infty)_b$ between  fibres of $\xi$ and fibres of $\xi_\infty$ suitably topologized. 
The projection is given by $f_{a,b} \mapsto  a \in X$ and so the fibre over a basepoint of $X$
 may   be identified with $\aut(F).$  Next, the authors define a fibration
 $$ \Phi \colon \Gamma(\xi \, \Box_1 \xi_\infty) \to \map(X, B\aut(F))$$
 with fibre $\Phi^{-1}(h) \approx \aut(\xi).$ We observe that the construction of $\Phi$  
 and the identification of the  fibre  both hold  in our case.  For the former, the key is \cite[Proposition 2.1]{BHMP} whose proof depends on the use of 
 exponential laws for functional fibrations which hold  for $X$ compact by \cite{BHP}.
 The identification of the fibre over $h$ with $\aut(\xi)$ requires that $\xi$ be an $\mathcal{F}$-fibration in the sense of \cite[Definition 2.1]{Maymemoir} where
 $\mathcal{F}$ is the category of ``fibres'' homotopic to $F.$  This fact is a 
  consequence of our assumption that $\xi$ is the pullback of a CW fibration (use \cite[Lemma 3.3 and Proposition 3.4]{Maymemoir}).   
 
 \begin{Lem} \label{lem:contract} The space $\Gamma(\xi \, \Box_1 \xi_\infty)$ is contractible.   
\end{Lem}
\begin{proof}
Observe that $ \xi \, \Box_1 \xi_\infty$   
is the pullback  of the fibration $$ \xi_\infty \, \Box_1 \xi_\infty \colon E_\infty \, \Box E_\infty \to B\aut(F)$$ by $h \colon X \to B\aut(F)$.  The fibration $\xi_\infty \, \Box_1 \xi_\infty$  is a CW fibration by Sch{\"o}n \cite{Schon} since both the base $B\aut(F)$ and the fibre $\aut(F)$ are   CW  since $F$ is finite CW. Thus $\Gamma(\xi \, \Box_1 \xi_\infty)$ is of CW type  \cite[Proposition 3.2]{KSS}. 
It  thus suffices to show $\Gamma(\xi \, \Box_1 \xi_\infty)$
is weakly contractible.  

For this, we apply Theorem \ref{thm:Gammalimit}.  Following the procedure described before this result, write  $X = \Invlim_j X_j$ with each $X_j$  a finite complex.   For $j \geq m$ as above,     let $\xi_j \colon E_j \to X_j$    be the corresponding fibration over $X_j$.  Let  
$\xi_j  \, \Box_1 \xi_\infty \colon E_j \, \Box  E_\infty \to X_j$ the associated construction.
Functoraility of this construction gives fibre maps      $\xi_j \, \Box_1 \xi_\infty \to \xi_k \, \Box_1 \xi_\infty $ for $j \geq k \geq m$ and compatible fibre maps $\xi  \, \Box_1 \xi_\infty \to \xi_j \, \Box_1 \xi_\infty$. We thus obtain a direct system $\Gamma(\xi_j \, \Box_1 \xi_\infty )$ of sections with compatible maps $\gamma_j \colon \Gamma(\xi_j \, \Box_1 \xi_\infty ) \to\Gamma(\xi  \, \Box_1 \xi_\infty )$ satisfying the conditions of Theorem \ref{thm:Gammalimit}.     By \cite[Proposition 3.1]{BHMP}, each space
$\Gamma(\xi_j \, \Box_1 \xi_\infty)$ is contractible.  Thus 
$$ \pi_n(\Gamma(\xi \, \Box_1 \xi_\infty))  \cong  \Dirlim_j \pi_n( \Gamma(\xi_j \, \Box_1 \xi_\infty )) = 0,$$ for $n \geq 0$, as needed. \end{proof}

 We obtain, as a consequence:
 
 \begin{proof}[Proof of Theorem \ref{introthm:aut}]  
  Let $X$ be a compact metric space, $F$ a finite CW complex and
 $h \colon B \to B\aut(F)$ a map.  Let $\xi \colon E \to X$ be the corresponding
 $F$-fibration.  By Lemma \ref{lem:contract}, the connecting map
 $$\delta \colon \Omega \map(X,  B\aut(F); h)) \to \Phi^{-1}(h) \approx \aut(\xi)$$
 in the  Barratt-Puppe sequence for $\Phi \colon \Gamma(\xi \,  \Box_1 \, \xi_\infty) \to \map(X, B\aut(X))$ is a homotopy equivalence. By \cite[Theorem 3.3]{BHMP}, $\delta$  is a multiplicative map and so gives the desired H-equivalence. 
 \end{proof}

Using Theorem \ref{introthm:aut} we obtain:

\begin{proof}[Proof of Theorem \ref{intro1}]
 Recall we are assuming $X$ is a simply connected, compact metric space
 and $ \xi \colon E \to X$ is the pullback of the universal $F$-fibration via $h \colon X \to B\aut(X).$  Since $X$ is simply connected,   $h$ lifts to a map $\tilde{h} \colon X \to B\aut(X)_\circ$ to the universal cover. Further, we have  an $H$-equivalence $$\Omega_\circ \map(X, B\aut(F); h) \simeq\Omega_\circ\map(X, B\aut(F)_\circ; \tilde{h})$$
 as these spaces are evidently weakly H-equivalent and both are  CW complexes \cite{Mil}. 
 The result now follows from Theorem \ref{thm:maplocalize} and Theorem \ref{introthm:aut}
 \end{proof}
 
\section{ Consequences in Rational Homotopy Theory}
\label{sec:applications}

There are good algebraic models for the rational homotopy theory of the monoid $\aut(\xi)$  and the space of sections $\Gamma(\zeta).$  
(See \cite{FLS} for the former and \cite{Hae, BFM} for the latter.)
However, these models require   finiteness and/or nilpotence conditions on the fibration 
and so are not directly applicable to the case of compact domain. In our main results above,    we   require that $\xi$ and $\zeta$ occur as the  pullback of a CW fibration. In this section, we observe  that, 
in certain special  circumstances,  the structure of the rationalization of these CW fibrations
allow for a description of the rational homotopy theory of  these grouplike spaces of maps. 

For the monoid $\aut(\xi)$,  the CW fibration in question is the universal $F$-fibration.  The following result is the $\PPP$-local version of \cite[Proposition 6.1]{BHMP}   in our setting. 

\begin{Thm} \label{thm:app1} Let $X$ be a  simply connected, compact  metric space and $F$ a finite complex.  Suppose the $\PPP$-localization of $B\aut(F)_\circ$ is a grouplike space.  Then, for all   fibrations
$\xi$  corresponding to maps in $[X, B\aut(F)]$, we have an H-equivalence
$$ (\aut(\xi)_\circ)_\PPP \simeq \map(X, (\aut(F)_\circ)_\PPP; 0)$$
where the latter is the space of null maps       into the grouplike space
$(\aut(F)_\circ)_\PPP$ which is a grouplike space  with pointwise multiplication.
 \end{Thm}

\begin{proof} The proof is the same as that given in \cite[Example 4.7]{KSS}. We give the details for completeness. Using the homotopy inverse for  $(B\aut(F)_\circ)_\PPP,$  we obtain a homotopy  equivalence   between   
$\map(X, (B\aut(F)_\circ)_\PPP; \tilde{h}_\PPP)$ and  $\map(X, (B\aut(F)_\circ)_\PPP; 0)$. Looping this equivalence and applying Theorem \ref{intro1} yields   H-equivalences:
$$(\aut(\xi)_\circ)_\PPP \simeq \Omega_\circ \map(X, (B\aut(F)_\circ)_\PPP; \tilde{h}_\PPP) \simeq
\Omega_\circ \map(X, (B\aut(F)_\circ)_\PPP; 0)$$
At the null component, we may pull $\Omega_\circ$  inside to get
$$ \Omega_\circ \map(X, (B\aut(F)_\circ)_\PPP; 0) \simeq 
\map(X, \Omega_\circ (B\aut(F)_\circ)_\PPP; 0) \simeq \map(X, (\aut(F)_\circ)_\PPP; 0),$$
as needed.
\end{proof}

The  condition on $(B\aut(F)_\circ)_\PPP$ is, of course, very strong and will be   rarely satisfied. However, as we explain now,  
there is a famous class of spaces  in rational homotopy theory, identified by Halperin in \cite{Hal}, that    conjecturally (and in many known cases)
do have this property after rationalization. In the rational case,  observe that it is sufficient to check  $(B\aut(F)_\circ)_\Q$ is an    H-space to apply Theorem \ref{thm:app1} since, rationally, all H-spaces are grouplike (see, e.g., \cite{Sch}). 

We say a space $F$ is an {\em $F_0$-space} if $F$ is a simply connected elliptic space (i.e., a finite complex with $\pi_*(F) \otimes \Q$  finite-dimensional) and satisfying
$H^{\mathrm{odd}}(F; \Q) =0$.      Halperin's conjecture (see \cite{Hal}) translated to our setting is:

\begin{Conj} {\em (Halperin}) \label{conj}  Let $F$ be an $F_0$-space. Then the monoid $\aut(F)_\circ$ has vanishing even degree rational homotopy groups. 
\end{Conj} 

Examples of $F_0$-spaces include products of even-dimensional spheres and complex 
projective spaces for which Conjecture \ref{conj}  is easily confirmed.  Homogeneous 
spaces $G/H$ of equal rank, compact pairs $H \subseteq G$ are also $F_0$-spaces and satisfy Conjecture \ref{conj}  by  \cite{ST}. 

We note that, if $F$ is an $F_0$-space satisfying   Conjecture \ref{conj}, then $(B\aut(F)_\circ)_\Q$ is a rational H-space since the Sullivan model of a nilpotent space with no odd rational homotopy has trivial differential.   Further, the (odd) rational homotopy groups of $ \aut(F)_\circ$ can be, in practice,  directly computed  from the minimal model $(\land V; d)$ of $F$ using Sullivan's identity 
$$\pi_{k}(\aut(F)_\circ) \otimes \Q \cong H_{k}(\Der(\land V; d)).$$
(See, e.g.,    \cite{Griv}.) 
Here the latter space is  the homology of the DG vector space of negative degree derivations of $(\land V; d)$.  

The following result extends \cite[Theorem 4]{FT} and \cite[Example 2.7]{FLS}. We write $\Check{H}^*(X; \Q)$ for rational \v{C}ech cohomology.

\begin{Thm} \label{thm:F} Let $X$ be a simply connected, compact metric space and $F$ an $F_0$-space satisfying Conjecture \ref{conj}. Let $h \colon X \to B\aut(F)$ be a map and $\xi$ the corresponding fibration over $X$. Then:
\begin{enumerate}
\item[(1)]  The rational H-homotopy type   of the monoid $\aut(\xi)_\circ$ is  independent of the classifying map $h$. \newline 
\item[(2)]   The monoid  $\aut(\xi)_\circ$ is rationally H-commutative and equivalent to a product of Eilenberg-Mac~Lane spaces with homotopy groups given by 
$$\pi_n(\aut(\xi)_\circ) \otimes \Q \cong \bigoplus_{k \geq n} \Check{H}^{k-n}(X; \Q) \otimes (\pi_{k}(\aut(F)_\circ) \otimes \Q)$$
for $n \geq 1.$
\end{enumerate}
\end{Thm}

\begin{proof}
The first statement follows directly from the preceding  paragraph     and 
Theorem \ref{thm:app1}. As for (2), we note $\pi_*(\aut(F)_\circ) \otimes \Q$ is oddly graded and so, for degree reasons,  does not admit any Samelson products. It follows
that $\aut(F)_\circ$ is rationally H-commutative    
\cite{Sch}. Thus $\aut(\xi)_\circ \simeq_\Q \map(X, \aut(F)_\circ; 0)$ is rationally H-commutative, as well  by   \cite[Theorem 4.10]{LPSS}.

The computation of the rational homotopy groups of  $\map(X, \aut(F); 0)$ reduces to the problem of computing homotopy groups of the space of maps into an 
Eilenberg-Mac~Lane space. Here we have the identity  of Thom \cite{Thom}:   $$\pi_q(\map(X, K(\pi, n)) \cong H^{n-q}(X; \pi)$$ which holds, with ordinary cohomology when $X$ is CW.  For
compact spaces,  the same identity holds with rational \v{C}ech cohomology 
by writing $X = \Invlim X_j$  and using the continuity of \v{C}ech cohomology
(see the proof of  \cite[Theorem 5.6]{LPSS}.)  
\end{proof}

We prove a related result for grouplike spaces of sections. We say a fibration $\xi \colon E \to X$
of connected CW complexes is {\em nilpotent} if the spaces $E$ and $X$ are nilpotent. 
In this case, $F$ is also nilpotent by \cite[Theorem II.3.12]{HMR} and  term-by-term $\PPP$-localization
gives a fibration $\xi_\PPP \colon E_\PPP \to X_\PPP$ with fibre $F_\PPP$ \cite[Proposition II.2.13]{HMR}. 
Note that the term-by-term $\PPP$-localization  
$\xi_\PPP \colon E_\PPP \to X_\PPP$ is not the same as the fibrewise $\PPP$-localization   $\xi_{(\PPP)} \colon E_0 \to X$. In fact,  we have:
\begin{Lem} \label{lem:dold}
Let $\xi \colon E \to X$ be a nilpotent fibration with term-by-term $\PPP$-localization $\xi_\PPP$ and fibrewise $\PPP$-localization $\xi_{(\PPP)}$.   Then there is a fibre-homotopy equivalence
 $$\xi_{(\PPP)}  \simeq \ell_X^{-1}(\xi_\PPP)$$ where $\ell_X \colon X \to X_\PPP$  is a $\PPP$-localization map. 
\end{Lem}
\begin{proof}
Using  the uniqueness theorem for fibrewise localization \cite[Theorem 6.1]{Ll} we obtain a fibrewise map 
$\xi_{(\PPP)} \to (\ell_X)^{-1}(\xi_\PPP)$ over $X$. By the 5-lemma and \cite[Theorem 3.3]{Dold},  this is  a fibre homotopy equivalence. 
\end{proof}

As a direct consequence, we have: 

\begin{Thm} \label{thm:app2}  Let $\zeta \colon E \to X$ be a fibrewise grouplike space over a compact metric space $X$ with fibre $G$ connected CW. Suppose $\zeta$ is the pullback of 
nilpotent fibration $\xi$ with $\PPP$-localization  $(\xi)_\PPP$ fibre-homotopically trivial.
Then  
$$ 
(\Gamma(\zeta)_\circ)_\PPP \simeq \map(X, G_\PPP; 0).$$
 
\end{Thm}

\begin{proof} 
By  the remarks preceding the  proof  of Theorem \ref{intro2}, the  fibrewise localization of $\zeta$ may be taken as the pullback of the fibrewise localization $\xi_{(\PPP)}$ of $\xi$.
Since the latter is the pullback of $\xi_{\PPP}$ which is, by hypothesis,  fibre-homotopically trivial,
we see $\zeta_{(\PPP)}$ is fibre-homotopically trivial, as well.  Now use Theorem \ref{intro2}. 
\end{proof}

We apply this last result to a natural generalization of the gauge group. Given a   group $G$ let $PG = G/Z(G)$ denote the projectification.  Given a map $h \colon X \to BPG$ write  $P_h \colon E \to X$ for the corresponding principal $PG$-bundle and $$\Pad(P_h) \colon E \times_{PG} G^{\ad}  \to X$$
the associated $G$-bundle where $PG$ acts on $G^{\ad} = G$ by the adjoint action.
This is a fibrewise group.  We call the group  of sections $\Gamma(\Pad(P_h))$ the {\em projective gauge group}.   
When $G = U(n)$, the projective gauge group   corresponds to the group of unitaries 
of a complex matrix bundle (see \cite[Example 3.7]{KSS}).  

The classification of rational H-types of projective gauge groups for $G$ a compact, connected Lie group and $X$ a compact metric space is given
by \cite[Theorem F]{KSS}.  The following result extends this to all connected Lie groups but at the expense of the H-structure.

\begin{Thm} \label{thm:P} Let $X$ be a compact metric space and  $G$ a connected Lie   group.  Then:
\begin{enumerate} 
\item[(1)] The rational homotopy type of $\Gamma(\Pad(P_h))_\circ$ is independent of   $h \in [X, BPG].$     \newline
\item[(2)]  The rational homotopy groups of $\Gamma(\Pad(P_h))_\circ$ are given by
$$ \pi_n(\Gamma(\Pad(P_h))_\circ) \otimes \Q  
 \cong\,\,  \bigoplus_{k \geq n} \Check{H}^{k-n}(X; \Q)
\, \otimes \, \left( \pi_k(G) \otimes \Q \right),$$
for $ n \geq 1.$   
\end{enumerate}

\end{Thm}
\begin{proof} 
We note that $\Pad(P_h)$ is the pullback by $h$ of the universal projective adjoint bundle
$\Pad(\eta_{PG}) \colon  EPG \times_{PG} G^{\ad} \to BPG$.  Here $\eta_{PG} \colon EPG \to BPG$ is the universal  principal $PG$-bundle.  By  \cite[Lemma 5.11]{KSS}, the total space $EPG \times_{PG} G^{\ad}$ is a nilpotent space.  (This result assumes $G$ is compact Lie.   However, all that is used there is that $Z(G)$ have vanishing higher rational homotopy groups, which is true for $G$ connected Lie.)  
Thus $\Pad(\eta_{PG})$ is nilpotent since $BPG$ is simply connected.  
Below we show the rationalization of $\Pad(\eta_{PG)})$ is fibre-homotopically trivial.
The result (1) then follows from Theorem \ref{thm:app2}. The result (2) is proved using   the arguments given
in the proof of Theorem \ref{thm:F} (2). 

To show $\Pad(\eta_{PG})_\Q \colon (EPG \times_{PG} G^{\ad})_\Q  \to (BPG)_\Q$
is fibre homotopically trivial, we first show the total space is an H-space. 
Let $EG \times_{G} G^{\ad}$ denote the total space of the universal $G$-adjoint bundle. 
Then $EG \times_{G} G^{\ad}$ is also a nilpotent space. In fact, we have 
$$ EG \times_{G} G^{\ad} \simeq \map(S^1, BG;0)$$
the free loop space of the classifying space (see, e.g., \cite[Lemma 9.1]{KSS}).
Since $BG$ is simply connected,  $\map(S^1, BG; 0)$ is nilpotent   by \cite[Theorem II.3.11]{HMR}.  Further, by this last result again,  
   $$ (EG \times_{G} G^{\ad})_\Q \simeq \map(S^1, (BG)_\Q;0).$$
Now since $G$ is a connected Lie group, it has oddly graded rational homotopy groups. It follows, as above, that    $(BG)_\Q$ is an  H-space.  Thus $\map(X, (BG)_\Q; 0)$ is an H-space, as well.    Next,  by \cite[Lemma 5.9]{KSS}, the natural map $$\pi \colon EG \times_{G} G^{\ad}
 \to  EPG \times_{PG} G^{\ad}  
$$ induces a surjection on rational homotopy groups. (Again, the lemma is stated for  $G$ compact Lie but the proof uses only that $Z(G)$ is rationally aspherical.) Thus, by \cite[Lemma 5.8]{LPSS}, 
since $ EG \times_{G} G^{\ad}$ is a rational H-space so is $EPG \times_{PG} G^{\ad}.$ We have shown that  $\Pad(\eta_{PG})$ is a sectioned, nilpotent fibration of rational H-spaces   with simply connected base.  By Lemma \ref{lem:K} below,   $(\Pad(\eta_{PG}))_\Q$ 
is fibre-homotopically trivial, as needed.
 \end{proof} 

\begin{Lem} \label{lem:K}
Let $\xi \:  E \to B $ 
be a fibration   of   nilpotent
 spaces with $E, B$ and the fibre $F$ each  a rational $H$-space of finite type. Suppose
 $B$ is
 simply connected and the linking
 homomorphism $ \partial \: \pi_k(B) \longrightarrow
 \pi_{k-1}(E)$
  is trivial after rationalization for all $k \geq 2$.  Then
 $\xi_\Q$ is   fibre homotopically  trivial.
 \end{Lem}
\begin{proof}
By  \cite[Theorem 4.6]{Hal},
$\xi$ is a    rational fibration  as in \cite[Definition 4.5]{Hal}.
In general, this means that
$\xi_\Q$ admits a Koszul-Sullivan model which is a sequence
$$
(\land V_3; d_3)  \stackrel{P}{\longrightarrow}  \land (V_2; d_2)
\stackrel{J}{\longrightarrow} (\land V_1; d_1).
$$
Here $(\land V_j; d_j)$ is a free DG algebra  over $\Q$ for $j=1,2,3$  and a Sullivan model for $F, E, B$,  respectively.   The maps  $P$
and $J$ are Sullivan models for the projection and fibre inclusion.
The models $(\land V_1;d_1)$ and $(\land V_3; d_3)$ are
minimal, meaning $d_j(V_j)$ is contained in the decomposables
of $\land V_j$ for $j= 1,3.$ Since $B$ and $F$ are assumed to be rational H-spaces, this forces $d_3 = d_1= 0.$ The model $(\land V_2; d_2)$ is not, in
general, minimal.
However, by  \cite[Theorem 4.12]{Hal}, the vanishing of the rational
linking homomorphism implies $(\land V_2; d_2)$ is a minimal model
for $E,$ as well. Since $H^*(E; \Q) \cong H^*(\land V_2; d_2)$
is free we conclude $d_2 = 0.$  Now we may directly obtain
a DG algebra map  $\widetilde{J} \: \land (V_2; 0) \rightarrow \land (V_3;0)$
with $\widetilde{J} \circ P  = 1_{\land V_3}$. The maps $\widetilde{J}, J$ induce an
isomorphism $A \: \land
(V_2; 0)
\rightarrow (\land V_1; 0) \otimes (\land V_3; 0)$  of DG algebras
satisfying $  \pi_2 \circ A = P$ where $\pi_2 \: (\land V_1; 0) \otimes (\land V_3; 0)
\rightarrow \land (V_3; 0)$ is the projection.
Using the correspondence   between (homotopy classes of) maps between minimal DGAs and maps
   between  rational spaces,  we obtain $\alpha \: F_\Q \times B_\Q
   \longrightarrow E_\Q,$ the needed fibre homotopy equivalence.
\end{proof}

\section{The Based Case   }
 \label{sec:final}

We deduce versions of our main results in the basepoint preserving cases.  First, let $\xi \colon E \to X$ be a fibration over a based space $X$ and let $F$ be the fibre over a fixed basepoint.  We denote
by $\aut^F(\xi)$ the submonoid of $\aut(\xi)$ consisting of equivalences inducing the identity on $F.$  This is the natural based version of $\aut(\xi).$ 
Note that $\aut^F(\xi)$ is the fibre over the identity map $1_F$ of the restriction map $\mathrm{res} \colon \aut(\xi) \to \aut(F).$ 
We have the following  based version of the identity (\ref{eq:Gottlieb}) in our setting. Write $\map_*(X, Y)$ for the space of basepoint perserving maps from $X$ to $Y$. 

\begin{Thm}    Let $X$ be a compact metric space, $F$ a finite CW complex and
 $h \colon B \to B\aut(F)$ a map.  Let $\xi \colon E \to X$ be the corresponding
 $F$-fibration.  Then there is an H-equivalence
 $$\aut^F(\xi) \simeq \Omega \map_*(X, B\aut(F); h).$$
 \end{Thm}
 \begin{proof}  
First, note that, by  \cite{Schon}, $\aut^F(\xi)$ is of CW type since $\aut(\xi)$ and $\aut(F)$. Now compare the long exact homotopy sequence of the restriction fibration above to that of loops on the evaluation fibration $\omega \colon \map(X, B\aut(F); h) \to B\aut(F)$ with fibre
$\map_*(X, B\aut(F); h).$ The result follows from Theorem \ref{introthm:aut} and the   $5$-lemma. 
 \end{proof}

As before, this gives a corresponding  localization result for $\aut^F(\xi)_\circ$:
\begin{Thm}
Let $X$ be a compact metric space, $F$ a finite CW complex and
 $h \colon B \to B\aut(F)$ a map.  Let $\xi \colon E \to X$ be the corresponding
 $F$-fibration. Then $\aut^F(\xi)_\circ$ is a nilpotent space  
and  we have an H-equivalence $$ (\aut^F(\xi)_\circ)_\PPP \simeq\Omega_\circ \map_*(X, (B\aut(F)_\circ)_\PPP; (\tilde{h})_\PPP ). $$ 
\qed
\end{Thm} 

Next, given a fibration $\zeta \colon E \to X$ write  $\Gamma_*(\zeta)$ for the space of basepoint preserving sections of $\zeta.$ We have:

\begin{Thm}   Let $\zeta \colon E \to X$ be a fibrewise grouplike space with connected, CW grouplike fibre $G$  and base $X$  a connected compact metric space. Suppose $\zeta$    is the pullback of a CW fibration.  Then $\Gamma_*(\zeta)_\circ$ is a nilpotent space and the map $$\Gamma_*(\zeta)_\circ \to \Gamma_*(\zeta_{(\PPP)})_\circ$$ induced by a fibrewise $\PPP$-localization $\zeta \to \zeta_{(\PPP)}$
 is a $\PPP$-localization map. 

 \end{Thm}

\begin{proof} In this case, the relevant evaluation fibration $\omega \colon \Gamma(\zeta)_\circ \to E$ with fibre $\Gamma_*(\zeta)_\circ$ is    not of use since  $E$ is not well-behaved under fibrewise localization.
However, we may repeat the entire argument  in the based case to achieve the needed result. We note that M{\" o}ller's theorem \cite[Theorem 5.3]{Mol}, for $X$ a finite complex,  is proved in the based setting. The rest of the argument thus proceeds as before.
\end{proof}

Finally, we remark that our applications Theorems \ref{thm:F} and \ref{thm:P} also hold in the respective based
settings but with ordinary \v{C}ech cohomology replaced by reduced \v{C}ech cohomology in the rational homotopy calculations.

\end{document}